\documentclass[lettersize,journal]{IEEEtran}
\usepackage{array}
\usepackage[caption=false,font=normalsize,labelfont=sf,textfont=sf]{subfig}
\usepackage{textcomp}
\usepackage{stfloats}
\usepackage{url}
\usepackage{verbatim}
\usepackage{array}
\usepackage{url}
\usepackage{amsmath,amsfonts}
\usepackage{amsthm}
\usepackage{tikz}
\usetikzlibrary{matrix}
\usepackage{import}
\usepackage{float}
\usepackage{mathtools}
\usepackage{cite}
\usepackage{graphicx}
\usepackage{bm}
\usepackage{hyperref}
\usepackage{newtxmath} 
\usepackage[noabbrev,nameinlink]{cleveref}
\usepackage{orcidlink} 

\usepackage[dvipsnames]{xcolor}
\definecolor{mycolor1}{RGB}{24, 92, 68} 

\hypersetup{
    colorlinks = true,
    linkcolor = mycolor1,   
    citecolor = mycolor1,   
    urlcolor = Blue          
}

\DeclareSymbolFont{AMSb}{U}{msb}{m}{n}
\DeclareSymbolFontAlphabet{\mathbb}{AMSb}
\DeclareMathOperator*{\diag}{diag}

\newtheorem{assumption}{Assumption}
\newtheorem{lemma}{Lemma}
\newtheorem{proposition}{Proposition}
\newtheorem{theorem}{Theorem}
\newtheorem{definition}{Definition}

\title{\LARGE \bf
Non-Locally Controllable but Trackable Magnetic Head Flagellated Swimmer
}

\author{Lucas Palazzolo$^{1,*}$ \orcidlink{0009-0002-4188-014X}, Mickaël Binois$^{2,\ddagger}$ \orcidlink{0000-0002-7225-1680} and  Laetitia Giraldi$^{1,\dagger}$\orcidlink{0000-0003-2684-0203}
\thanks{$^{1}$Université Côte d'Azur, Inria, Calisto team, Sophia Antipolis, France
        }%
\thanks{$^{2}$Université Côte d'Azur, Inria, Acumes team, CNRS, LJAD, Sophia Antipolis, France}%
\thanks{$^{*}$ {\tt\small lucas.palazzolo@inria.fr}}%
\thanks{$^{\dagger}$ {\tt\small laetitia.giraldi@inria.fr}}%
\thanks{$^{\ddagger}$ {\tt\small mickael.binois@inria.fr}}}

\begin{document}

\maketitle
\thispagestyle{empty}
\pagestyle{empty}

\begin{abstract}
Unlike macroscopic swimmers, microswimmers operate in a low-Reynolds-number regime dominated by viscous forces. This paper investigates the controllability of a magnetic microswimmer composed of a spherical magnetic head and an elastic, non-magnetic flagellum. The swimmer evolves in a Stokes flow and is modeled using the resistive force theory. We prove that, under planar motion, the system is not small-time locally controllable and numerically identify regions that remain inaccessible. Nevertheless, simulations show that trajectory tracking can still be achieved via Bayesian optimization, though it requires large-amplitude transverse deformations.
\end{abstract}


\section{Introduction}
Microswimmers are microorganisms capable of self-propulsion in fluid environments, such as bacteria or bio-inspired microrobots. The application domain of microrobotics covers several fields, notably healthcare, where the objective is to deliver drugs directly to target cells for cancer treatment \cite{medical_robots_cancer_2020, zhang2023}, or environmental engineering, for instance in water purification \cite{soler2013}. \\

A wide range of microrobots exists, employing diverse propulsion mechanisms, such as flagellar motion \cite{zhang2023}, light-based actuation \cite{varun2022}, chemical gradients \cite{zhuang2017}, or external magnetic fields \cite{Dreyfus2005, oulmas2017}. The latter category is the focus of this work. These swimmers consist of a spherical magnetic head attached to a non-magnetic elastic flagellum and operate in a Stokes flow regime. Due to the elasticity of the flagellum, the \textit{scallop theorem} \cite{purcell} is no longer an obstruction, as elastic deformations naturally generate non-reciprocal shape changes. By discretizing the swimmer into $N$ rigid links, the \textit{Resistive Force Theory} (RFT) provides a simple and well-suited framework, leading to what is known as an \textit{N-link swimmer} \cite{giraldi2013}. The resulting dynamics can then be described by an Ordinary Differential Equation (ODE) \cite{moreau2019}.\\

The question of controllability has attracted considerable attention from researchers studying microswimmer models. Using a sub-Riemannian geometric framework, numerous studies have investigated the controllability of active particles in Stokes flow, including squirmers \cite{Tucsnak2016, loheac2020}, spherical swimmers \cite{alouges2013, giraldi2015}, and deformable flagellated swimmers \cite{zoppello2025}. However, only a few works provide non-controllability results. A notable contribution is the series of papers by C.~Moreau, which analyze two-dimensional systems composed of elastic and magnetic flagella \cite{moreau2018, moreau2019}. These results were later generalized to broader classes of equations in \cite{beauchard2018,moreau2024}.\\

In this paper, we propose a three-dimensional formulation based on the RFT to model the fluid–structure interaction, where the flagellum elasticity is represented by spring-like connections between consecutive links \cite{faris2020, zoppello2025}. These magnetic microswimmers show strong potential for biomedical applications \cite{oulmas2017, faris2020}. Building on the controllability framework of \cite{moreau2024}, we establish our main result showing the lack of local controllability of a magnetic swimmer composed of a magnetic head and an elastic flagellum, restricted to planar motion. Extending the analysis to fully three-dimensional motion would require studying an affine system with three controls, a considerably more complex and still open problem. Numerical simulations show regions that remain unreachable near the equilibrium under small oscillatory control inputs. Despite these negative controllability results, trajectory tracking is still feasible, though it requires large-amplitude motions in directions transverse to the desired displacement.\\

The paper is organized as follows. \Cref{sec:1} presents the mathematical modeling of the elastic flagellated microswimmer with a magnetic head. \Cref{sec:2} introduces the main definitions and theoretical results related to \textit{Small Time Local Controllability} (STLC). \Cref{sec:3} presents the main result of the paper concerning the $2$-link swimmer. In \Cref{sec:4}, the proof of the main result is detailed in two parts. \Cref{sec:5} provides numerical results illustrating the behavior near the equilibrium state and trajectory tracking using Bayesian optimization. Finally, \Cref{sec:6} summarizes and concludes the paper.

\section{Mathematical Modeling}\label{sec:1}
This section describes the mathematical model of the elastic microswimmer with a magnetic head. We first detail its parametrization and discretization into $N$ links, followed by the derivation of the governing dynamic equations.

\subsection{Swimmer's parametrization}
The swimmer, illustrated in \Cref{fig:nlinks_illu}, consists of a spherical head of radius $r$, centered at $\boldsymbol{X}\in\mathbb{R}^3$, and a flagellum of length $L$ discretized into $N$ rigid links of length $l=L/N$ connecting $N+1$ points $\boldsymbol{X}^i$. The laboratory frame is $\mathscr{R} = (0_{\mathbb{R}^3}, \boldsymbol{e}_1, \boldsymbol{e}_2, \boldsymbol{e}_3)$, and the head frame is $\mathscr{R}^h = (\boldsymbol{X}, \boldsymbol{e}_1^h, \boldsymbol{e}_2^h, \boldsymbol{e}_3^h)$, with orientation given by the rotation matrix $R^h \in SO(3)$. Using Tait--Bryan angles $(\theta_x, \theta_y, \theta_z)$ with the $ZYX$ convention,
\begin{equation*}
R^h = R_x(\theta_x) R_y(\theta_y) R_z(\theta_z),
\end{equation*}
where $R_x$, $R_y$, and $R_z$ are the standard rotation matrices about the $x$-, $y$-, and $z$-axes, respectively. The flagellum attaches at $\boldsymbol{X}^1 = \boldsymbol{X} - r\boldsymbol{e}_1^h$, and the subsequent points are given by
\begin{equation*}
\boldsymbol{X}^i = \boldsymbol{X}^1 - l\sum_{k=1}^{i-1}\boldsymbol{e}_1^k, \quad i=2,\ldots,N+1.
\end{equation*}
Each link direction $\boldsymbol{e}_1^i$ is parameterized in $\mathscr{R}^h$ by spherical angles $(\phi_y^i, \phi_z^i)$, leading to
\begin{equation*}
R^i = R_y(\phi_y^i) R_z(\phi_z^i), \qquad \boldsymbol{e}_1^i = R^h R^i \boldsymbol{e}_1.
\end{equation*}
For $s\in[0,l]$, the position along link $i$ is
\begin{equation*}
\boldsymbol{x}^i(s) = \boldsymbol{X}^i - sR^hR^i\boldsymbol{e}_1,
\end{equation*}
and its time derivative reads
\begin{equation*}
\dot{\boldsymbol{x}}^{i}(s) = \dot{\boldsymbol{X}}^i + s [R^hR^i\boldsymbol{e}_1]^{\times}\boldsymbol{\Omega}^h + s R^h[R^i\boldsymbol{e}_1]^{\times}\boldsymbol{\Omega}^i.
\end{equation*}

\begin{figure}[t]
    \centering
    \includegraphics[width=1\linewidth]{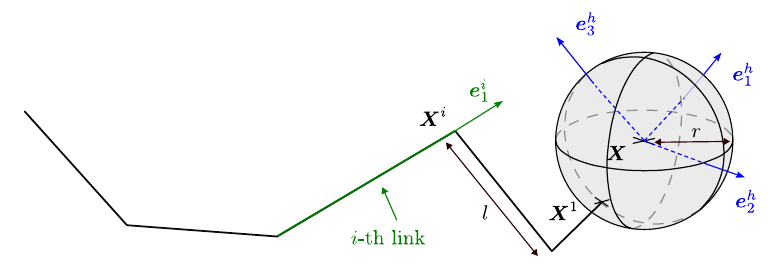}
    \caption{3D $N$-link model of the swimmer. The head frame is $\mathscr{R}^h = (\boldsymbol{X}, \boldsymbol{e}_1^h, \boldsymbol{e}_2^h, \boldsymbol{e}_3^h)$; each link $i$ is oriented along $\boldsymbol{e}_1^i$ of length $l$.}
    \label{fig:nlinks_illu}
\end{figure}

\subsection{Swimmer dynamics}
At low Reynolds number, inertia is negligible, and total force and torque balance yield
\begin{equation*}
\begin{cases}
\boldsymbol{F}^{\text{hydro}}_{\text{head}} + \displaystyle\sum_{i=1}^N \boldsymbol{F}^{\text{hydro}}_i = \boldsymbol{F}^{\text{ext}},\\[4pt]
\boldsymbol{T}^{\text{hydro}}_{\text{head}} + \displaystyle\sum_{i=1}^N \boldsymbol{T}^{\text{hydro}}_{i,\boldsymbol{X}} = \boldsymbol{T}^{\text{ext}}_{\text{head}},\\[4pt]
\displaystyle\sum_{i=j}^N \boldsymbol{T}^{\text{hydro}}_{i,\boldsymbol{X}^j} = \boldsymbol{T}_j^{\text{ext}}, \quad j=1,\ldots,N.
\end{cases}
\end{equation*}

\noindent\textbf{Hydrodynamic forces:}
The head experiences a viscous drag
\begin{equation*}
\boldsymbol{F}^{\text{hydro}}_{\text{head}} = -rR^h D^h (R^h)^\top \dot{\boldsymbol{X}}, \qquad
\boldsymbol{T}^{\text{hydro}}_{\text{head}} = -r^3k_r\boldsymbol{\Omega}^h,
\end{equation*}
with $D^h = \mathrm{diag}(k^h_\parallel, k^h_\perp, k^h_\perp)$.  
For each link, the local viscous force density from RFT is
\begin{equation*}
\boldsymbol{f}_i^{\text{hydro}}(s) = -R^h \tilde{D}^i (R^h)^\top \dot{\boldsymbol{x}}^i(s),
\end{equation*}
where $\tilde{D}^i = R^i D^i (R^i)^\top$, $D^i=\mathrm{diag}(k^i_\parallel, k^i_\perp, k^i_\perp)$. Integrating along the link gives $\boldsymbol{F}_i^{\text{hydro}}$ and $\boldsymbol{T}^{\text{hydro}}_{i,\boldsymbol{x}^0}$. In the following, we assume identical coefficients for all links.

\noindent\textbf{External torques:}
Elastic torques between links follow discrete beam theory,
\begin{equation*}
\boldsymbol{T}^{\text{el}}_i = k_{\text{el}} (\boldsymbol{e}_1^i \times \boldsymbol{e}_1^{i-1}),
\end{equation*}
and the magnetic actuation torque is
\begin{equation*}
\boldsymbol{T}^{\text{mag}} = \boldsymbol{M} \times \boldsymbol{B}, \qquad \boldsymbol{M}=m\boldsymbol{e}_1^h, \quad \boldsymbol{B}(t)=[u_1,u_2,u_3]^\top.
\end{equation*}
Thus, $\boldsymbol{F}^{\text{ext}}=0$, $\boldsymbol{T}^{\text{ext}}_{\text{head}}=-\boldsymbol{T}^{\text{mag}}$, and $\boldsymbol{T}^{\text{ext}}_j=-\boldsymbol{T}^{\text{el}}_j$.

\noindent\textbf{Final form:}
Neglecting rotation about $\boldsymbol{e}_1^i$, only the components orthogonal to this axis are retained. The full dynamic system involves the state
\begin{equation*}
\boldsymbol{p}=(\boldsymbol{X},\boldsymbol{\Theta},\boldsymbol{\Phi})^\top,\quad
\boldsymbol{B}=(u_1,u_2,u_3),
\end{equation*}
with $\boldsymbol{X}\in\mathbb{R}^3$, $\boldsymbol{\Theta}\in[0,2\pi]^3$, and $\boldsymbol{\Phi}\in[0,2\pi]^{2N}$.  We now restrict our analysis to the planar motion of the swimmer. Without loss of generality, we set $z = 0$ and $u_3 \equiv 0$. The study of controllability for the full three-control system is considerably more intricate and remains an open research problem. Under the planar assumption, only the coordinates associated with the variables $x$, $y$, $\theta_z$, and $\phi_z^i$ are considered. Consequently, we have
\begin{equation*}
z = \theta_x = \theta_y = \phi_y^i = u_3 = 0.
\end{equation*} The system can be compactly written as
\begin{equation}\label{eq:noninvertsyst}
    (AQB)\, \dot{\boldsymbol{p}} = \boldsymbol{G}_0 + u_1 \boldsymbol{G}_1 + u_2 \boldsymbol{G}_2.
\end{equation}
Finally, we obtain
\begin{equation}\label{eq:systplan}
\dot{\boldsymbol{p}} = 
\boldsymbol{F}_0(\boldsymbol{p})
+ u_1\, \boldsymbol{F}_1(\boldsymbol{p})
+ u_2\, \boldsymbol{F}_2(\boldsymbol{p}),
\end{equation}
where the vector fields $\boldsymbol{F}_i$ are obtained by inverting the matrix $AQB$. Due to the complexity of the determinant of $AQB$ (see code in the Supplemental Material), we assume in the following that the parameters are chosen such that the determinant never vanishes.

\section{Theoretical Framework}\label{sec:2}
In this section, we present the general context for two control-affine systems, along with the assumptions required for this work. Notations and definitions are introduced, and relevant theorems from \cite{moreau2024}, which are used to prove our main results, are recalled.

\subsection{Preliminaries}

Let
\begin{equation}\label{eq:syst}
    \dot{\boldsymbol{p}} = \boldsymbol{F}_0(\boldsymbol{p}) + u_1 \boldsymbol{F}_1(\boldsymbol{p}) + u_2 \boldsymbol{F}_2(\boldsymbol{p}),
\end{equation}
be a general control-affine system where the vector fields $\boldsymbol{F}_0,\boldsymbol{F}_1,\boldsymbol{F}_2$  are real-analytic on $\mathbb{R}^n$ (or in a neighborhood of an equilibrium point) and $u_1,u_2$ are control functions in $L^{\infty}([0,T];\,\mathbb{R}^m)$ for some $T>0$, with $(n,m)\in\mathbb{N}^2$. Denote by $\boldsymbol{p}^{\text{eq}}$ an equilibrium state. We consider special assumptions about the control-affine system \eqref{eq:syst}. First of all, the vector fields $\boldsymbol{F}_0$ and $\boldsymbol{F}_1$ vanish at the equilibrium while $\boldsymbol{F}_2$ does not.
\begin{assumption}\label{assumption1} 
The vector fields satisfy
\begin{equation*}
\boldsymbol{F}_0(\boldsymbol{p}^{\mathrm{eq}})=0_{\mathbb{R}^n},\quad 
\boldsymbol{F}_1(\boldsymbol{p}^{\mathrm{eq}})=0_{\mathbb{R}^n},\quad
\boldsymbol{F}_2(\boldsymbol{p}^{\mathrm{eq}})\neq0_{\mathbb{R}^n}.
\end{equation*}
\end{assumption}
The second assumption considered is on the gradient of the first component of the vector fields $\boldsymbol{F}_0$ and $\boldsymbol{F}_1$.
\begin{assumption}\label{assumption2}At the equilibrium $\boldsymbol{p}^{\mathrm{eq}}$, the gradients of the first components of $\boldsymbol{F}_0$ and $\boldsymbol{F}_1$ are aligned with the first coordinate axis, i.e.,
\begin{equation*}
\nabla F_0^{(1)}(\boldsymbol{p}^{\mathrm{eq}}),\; \nabla F_1^{(1)}(\boldsymbol{p}^{\mathrm{eq}}) \in \mathrm{Span}(\boldsymbol{e}_1).
\end{equation*}
\end{assumption}

\begin{assumption}\label{assumption3}
At the equilibrium $\boldsymbol{p}^{\mathrm{eq}}$, the first component of $\boldsymbol{F}_2$ vanishes, i.e.,
\begin{equation*}
\boldsymbol{F}_2(\boldsymbol{p}^{\mathrm{eq}})=
\begin{bmatrix}
0 & a_2 & \ldots & a_n
\end{bmatrix}^\top,
\end{equation*}
where $a_i\in\mathbb{R}$ for all  $i\in\{2,\dots,n\}$.
\end{assumption}

\subsection{Notations}
Let $\boldsymbol{f}$ and $\boldsymbol{g}$ be real analytic vector fields. Their Lie bracket is denoted by $[\boldsymbol{f},\boldsymbol{g}]$ and given by 
\begin{equation*}
    [\boldsymbol{f},\boldsymbol{g}]:= (\boldsymbol{f}\cdot\nabla)\boldsymbol{g}-(\boldsymbol{g}\cdot\nabla)\boldsymbol{f}.
\end{equation*}
We will use some condensed notations (as in \cite{moreau2024}) for some Lie brackets. Let a family of real analytic vector fields $(\boldsymbol{F}_i)_{i=0}^{2}$, we define:
\begin{align*}
    \boldsymbol{F}_{ij} &:= [\boldsymbol{F}_i, \boldsymbol{F}_j],\\
    \boldsymbol{F}_{ijk}&:=\left[\boldsymbol{F}_i, [\boldsymbol{F}_j, \boldsymbol{F}_j]\right],\\
    \boldsymbol{F}_{ij,kl}&:=\left[[\boldsymbol{F}_i, \boldsymbol{F}_j], [\boldsymbol{F}_k, \boldsymbol{F}_l\right],\\
    \boldsymbol{F}_{ij, klm} &:= \left[[\boldsymbol{F}_{i},\boldsymbol{F}_j],[\boldsymbol{F}_k,[\boldsymbol{F}_l, \boldsymbol{F}_m]]\right].
\end{align*}

\subsection{Definitions}
For $\delta>0$, and $\boldsymbol{p}\in\mathbb{R}^n$, let $B(\boldsymbol{p}, \delta)$ be the open ball centered at $\boldsymbol{p}$ with radius $\delta$. 

\begin{definition}[\cite{beauchard2018}, $W^{k,\infty}$-STLC] The control system \eqref{eq:syst} is $W^{k,\infty}$-STLC at the equilibrium $(\boldsymbol{p}^{\text{eq}}, \boldsymbol{u}^{\text{eq}})$ if, for every $\varepsilon>0$, $\varepsilon'>0$, there exists $\delta>0$ such that, for every $\boldsymbol{p}_0$, $\boldsymbol{p}_1$ in $B(\boldsymbol{p}^{\text{eq}}, \delta)$, there exists a control $\boldsymbol{u}(\cdot)$ in $W^{k,\infty}([0,\varepsilon], \mathbb{R}^m)$ such that the solution $\boldsymbol{p}(\cdot):[0,\varepsilon]\to \mathbb{R}^n$ of the control system \eqref{eq:syst} with initial condition $\boldsymbol{p}(0)=\boldsymbol{p}_0$ satisfies $\boldsymbol{p}(\varepsilon)=\boldsymbol{p}_1$ and 
\begin{equation*}
    \|\boldsymbol{u}-\boldsymbol{u}^{\text{eq}}\|_{W^{k,\infty}([0,\varepsilon];\,\mathbb{R}^m)}\leq \varepsilon'
\end{equation*}
\end{definition}
\medskip 
Sometimes, we refer to \textit{STLC} for the case $k = 0$, since $W^{0,\infty}$-STLC coincides with the standard notion of STLC (i.e., $L^{\infty}$-STLC).
\medskip
\begin{definition}[\cite{moreau2024}, B-STLC]
The control system \eqref{eq:syst} is B-STLC at the equilibrium $(\boldsymbol{p}^{\text{eq}}, \boldsymbol{u}^{\text{eq}})$ if there exists $\alpha>0$ such that, for every $\varepsilon>0$, there exists $\delta>0$ such that, for every $\boldsymbol{p}_0$, $\boldsymbol{p}_1$ in $B(\boldsymbol{p}^{\text{eq}}, \delta)$, there exists a control $\boldsymbol{u}(\cdot)$ in $L^{\infty}([0,\varepsilon],\, \mathbb{R}^m)$ such that the solution $\boldsymbol{p}(\cdot):[0,\varepsilon]\to \mathbb{R}^n$ of the control system \eqref{eq:syst} with initial condition $\boldsymbol{p}(0)=\boldsymbol{p}_0$ satisfies $\boldsymbol{p}(\varepsilon)=\boldsymbol{p}_1$ and 
\begin{equation*}
    \|\boldsymbol{u}-\boldsymbol{u}^{\text{eq}}\|_{L^{\infty}([0,\varepsilon];\,\mathbb{R}^m)}\leq \alpha
\end{equation*}
\end{definition}

\subsection{Theorems}
Given the model system of the swimmer under consideration \eqref{eq:systplan}, the vector fields 
$\boldsymbol{F}_1$ and $\boldsymbol{F}_2$ (as well as the equilibrium controls 
$u_1^{\text{eq}}$ and $u_2^{\text{eq}}$) in \cite{moreau2024} must be inverted. 
We now assume that the equilibrium state is located at the origin. 
Let $\mathcal{R}_1$ denote the set of all iterated Lie brackets of 
$\boldsymbol{F}_0$ and $\boldsymbol{F}_2$ in which $\boldsymbol{F}_2$ appears at most once, 
and let $R_1$ be the subspace of $\mathbb{R}^n$ spanned by the values at $0$ 
of the elements of $\mathcal{R}_1$.
\begin{theorem}[\cite{moreau2024}, Theorem 3.2]\label{theorem32}
Consider the system \eqref{eq:syst} under \Cref{assumption1}. Assume $\boldsymbol{F}_{202}(0)\notin R_1$.
\begin{enumerate}
    \item If $\boldsymbol{F}_{202}(0)\in R_1 + \mathrm{Span}(\boldsymbol{F}_{212}(0))$, let $\beta\in\mathbb{R}$ be such that 
    \begin{equation*}
        \boldsymbol{F}_{202}(0)+\beta \boldsymbol{F}_{212}(0)\in R_1.
    \end{equation*}
    Then, for any $u_1^{\text{eq}}$ such that $u_2^{\text{eq}}\neq\beta$, system \eqref{eq:syst} is not STLC at $(0, ( u_1^{\text{eq}},0))$.
    \item If $\boldsymbol{F}_{202}(0)\notin R_1 + \mathrm{Span}(\boldsymbol{F}_{212}(0))$, then, for any $u_1^{\text{eq}} \in \mathbb{R}$, system \eqref{eq:syst} is not B-STLC at $(0,(u_1^{\text{eq}}, 0)).$
\end{enumerate}
\end{theorem}
If we are not in the case covered by \Cref{theorem32}, we must consider Lie brackets of higher order. For a given value of the equilibrium control $u_1^{\text{eq}}$, we define the mapping 
$D_{u_1^{\text{eq}}}: \mathbb{R}^2 \to \mathbb{R}^n$ as follows:
\begin{multline}\label{eq:Deqlamb}
    D_{u_1^{\text{eq}}}(\lambda_1, \lambda_2)=\lambda_1^2\left(\boldsymbol{F}_{02,002}(0)-u_1^{\text{eq}}\boldsymbol{F}_{02,102}(0)\right)\\+\lambda_2^2\left(\boldsymbol{F}_{12,012}(0)-u_1^{\text{eq}}\boldsymbol{F}_{12,112}(0)\right)\\ - \lambda_1\lambda_2 \Big(\boldsymbol{F}_{12,002}(0)+\boldsymbol{F}_{02,012}(0)\\-u_1^{\text{eq}}(\boldsymbol{F}_{12,102}(0)+\boldsymbol{F}_{02,112}(0))\Big).
\end{multline}
We define the condition $C(Q)$ such that 
\begin{equation}\label{eq:CQ}
C(Q) \Leftrightarrow 
\left\{
\begin{aligned}
&~~\text{there exists a linear form } 
\varphi: \mathbb{R}^n \to \mathbb{R}, 
\text{ whose}\\
&~ \text{restriction to } Q \text{ is zero, and such that the} \\
&~\text{quadratic form } 
(\lambda_1,\lambda_2) \mapsto 
\left\langle \varphi, D_{u_1^{\text{eq}}}(\lambda_1,\lambda_2) \right\rangle\\ 
&~\text{ is positive definite.}
\end{aligned}
\right.
\end{equation}
Let $R'$ denote the subspace of $\mathbb{R}^n$ such that 
\begin{equation*}
    R'=\mathrm{Span}\left(\boldsymbol{F}_{02,102}(0), \boldsymbol{F}_{12,112}(0), \boldsymbol{F}_{12,102}(0), \boldsymbol{F}_{02,112}(0)\right).
\end{equation*}
\begin{theorem}[\cite{moreau2024}, Theorem 3.8]\label{theorem38}
Consider the system \eqref{eq:syst} under \Cref{assumption1}. Assume $\boldsymbol{F}_{202}(0)\in R_1$, $\boldsymbol{F}_{212}(0)\in R_1$ and $\boldsymbol{F}_{12,02}(0)\in R_1$, and let $u_1^{\text{eq}}\in\mathbb{R}$. Then, with condition $C(Q)$ defined in \eqref{eq:CQ}, we have, at $(0,(u_1^\text{eq},0))$:
\begin{enumerate}
    \item If $C(R_1)$ holds, system \eqref{eq:syst} is not $W^{1,\infty}$-STLC.
    \item If $C(R_1+R')$ holds, system \eqref{eq:syst} is not $(W^{1,\infty},B)$-STLC.
\end{enumerate}
\end{theorem}

\section{Main Result}\label{sec:3}
We present the main result of this paper, which concerns the local uncontrollability in small time of the magnetic-head $2$-link swimmer. A related case was previously studied in \cite{moreau2019}, where a swimmer with $N$ magnetic links in two dimensions was shown to be controllable only under very specific conditions. More precisely, when $\boldsymbol{F}_{202}$ does not vanish at equilibrium, the system is not B-STLC. As indicated in \cite{moreau2024}, this arises from obstructions due to third-order Lie brackets. If $\boldsymbol{F}_{202}$ vanishes at equilibrium, the system is not ($W^{1,\infty}$, B)-STLC.\footnote{This result is weaker than the previous one. To obtain a stronger result, one would need to consider the space $R''$ defined in \cite{moreau2024}, which is generally too complex to handle in practice.} In this case, the obstruction comes from fifth-order Lie brackets.  Therefore, when only the head is magnetized, the swimmer loses the possibility of being locally controllable in small time, in contrast to the case studied in \cite{moreau2019}.\\

\noindent Without loss of generality, we take the equilibrium state at the origin. Indeed, the problem is invariant under translation due to the absence of the variables $x$ and $y$ in the dynamics. Moreover, it is also invariant under rotation, since a simultaneous rotation of the system and the control inputs leaves the dynamics unchanged. The following computations were carried out using the symbolic computation software \texttt{Wolfram}. We obtain that the Lie bracket $\boldsymbol{F}_{202}(0)$ can be expressed as

\begin{equation}\label{eq:F202}
    \boldsymbol{F}_{202}(0) = \alpha \boldsymbol{e}_1,
\end{equation}
where $\alpha$ is a constant (given in \eqref{eq:alpha}) depending on $r$, $l$, and the coefficients associated with the hydrodynamic forces.
\begin{theorem}\label{theorem:nonbstlc}
Consider the system \eqref{eq:systplan} for a $2$-link swimmer.
\begin{enumerate}
    \item If $\alpha \neq 0$, then the $2$-link swimmer with magnetic head is not B-STLC at the equilibrium.
    \item If $\alpha = 0$, then the $2$-link swimmer with magnetic head is not ($W^{1,\infty}$, B)-STLC.
\end{enumerate}
\end{theorem}

\section{Proof}\label{sec:4}
The proof of \Cref{theorem:nonbstlc} is constructed in two steps. First, under the given assumptions, we deduce that the space $R_1$ is contained in $\mathrm{Span}(\boldsymbol{e}_2, \ldots, \boldsymbol{e}_n)$, as stated in \Cref{prop1}. Then, using the symbolic computation software \texttt{Wolfram}, we verify the required conditions and apply \Cref{theorem32} and \Cref{theorem38} to establish the result for the $2$-link swimmer.

\subsection{\texorpdfstring{Property of the $R_1$ Space}{Property of the R1 Space}}

\begin{lemma}\label{lemma1}
 Let us consider the system \eqref{eq:syst} under the \Cref{assumption1}. Then, any iterated Lie brackets that doesn't contain $\boldsymbol{F}_2$ vanished at the equilibrium. 
 \end{lemma}
\begin{proof}
We proceed by induction. Let $H_k$ denote the property:  
\emph{“For any iterated Lie brackets of length $k$ that do not contain $\boldsymbol{F}_2$, the bracket vanishes at $\boldsymbol{p}^{\text{eq}}$.”}

\medskip
\noindent Initialization:
For $k = 1$, by \Cref{assumption1}, we have  $\boldsymbol{F}_0(\boldsymbol{p}^{\text{eq}}) = \boldsymbol{F}_1(\boldsymbol{p}^{\text{eq}}) = 0_{\mathbb{R}^n}.$

\medskip
\noindent Induction step:
Let $k \in \mathbb{N}$ and assume that $H_k$ holds.  
Consider an iterated Lie bracket $\boldsymbol{F}$ of length $k+1$ that does not contain $\boldsymbol{F}_2$.  
We can write it as
\begin{equation*}
\boldsymbol{F} = [\boldsymbol{g}, \boldsymbol{h}],
\end{equation*}
where $\boldsymbol{g}$ and $\boldsymbol{h}$ are Lie brackets of length less than or equal to $k$.  
By definition of the Lie bracket, we have
\begin{equation*}
\boldsymbol{F} = (\boldsymbol{g}\cdot\nabla) \boldsymbol{h} - (\boldsymbol{h}\cdot\nabla)\, \boldsymbol{g}.
\end{equation*}
By the induction hypothesis, $\boldsymbol{g}(\boldsymbol{p}^{\text{eq}}) = \boldsymbol{h}(\boldsymbol{p}^{\text{eq}}) = 0_{\mathbb{R}^n}$, and thus  
\begin{equation*}
\boldsymbol{F}(\boldsymbol{p}^{\text{eq}}) = 0_{\mathbb{R}^n}.
\end{equation*}
Therefore, $H_{k+1}$ holds, which completes the proof.
\end{proof}

\begin{lemma}\label{lemma2}
Consider the system \eqref{eq:syst} under \Cref{assumption1} and \Cref{assumption2}. Then, the gradient of the first component of any iterated Lie bracket that does not contain $\boldsymbol{F}_2$ belongs to $\mathrm{Span}(\boldsymbol{e}_1)$.
\end{lemma}
\begin{proof}
We proceed by induction. Let $H_k$ denote the property:  
\emph{“For any iterated Lie bracket of length $k$ that does not contain $\boldsymbol{F}_2$, the gradient of its first component belongs to $\mathrm{Span}(\boldsymbol{e}_1)$.”}

\medskip
\noindent Initialization:
For $k = 1$, by \Cref{assumption2}, we have $\nabla F_0^{(1)}(\boldsymbol{p}^{\text{eq}}), \ \nabla F_1^{(1)}(\boldsymbol{p}^{\text{eq}}) \in \mathrm{Span}(\boldsymbol{e}_1).$

\medskip
\noindent Induction step:  
Let $k \in \mathbb{N}$ and assume that $H_k$ holds.  
Consider an iterated Lie bracket $\boldsymbol{F}$ of length $k+1$ that does not contain $\boldsymbol{F}_2$.  
We can write it as
\begin{equation*}
\boldsymbol{F} = [\boldsymbol{g}, \boldsymbol{h}],
\end{equation*}
where $\boldsymbol{g}$ and $\boldsymbol{h}$ are Lie brackets of length less than or equal to $k$.  
For simplicity, we omit the evaluation at $\boldsymbol{p}^{\text{eq}}$ in what follows.  
By definition of the Lie bracket, for all $i \in \{2, \ldots, n\}$, we have
\begin{multline*}
\partial_i F^{(1)} = \sum_{j=1}^n \Big( 
\partial_i g^{(j)} \, \partial_j h^{(1)} 
- \partial_i h^{(j)} \, \partial_j g^{(1)} 
\\+ g^{(j)} \, \partial_i \partial_j h^{(1)} 
- h^{(j)} \, \partial_i \partial_j g^{(1)}
\Big).
\end{multline*}
By \Cref{lemma1}, both $\boldsymbol{g}$ and $\boldsymbol{h}$ vanish at the equilibrium, and therefore the last two terms are zero. Hence,
\begin{multline*}
\partial_i F^{(1)} = 
\partial_i g^{(1)} \, \partial_1 h^{(1)} 
- \partial_i h^{(1)} \, \partial_1 g^{(1)} 
\\+ \sum_{j=2}^n \left(
\partial_i g^{(j)} \, \partial_j h^{(1)} 
- \partial_i h^{(j)} \, \partial_j g^{(1)}
\right).
\end{multline*}
By the induction hypothesis, all partial derivatives $\partial_i$ of the first component with $i \ge 2$ vanish at equilibrium.  
Therefore, $\partial_i F^{(1)}$ also vanishes at equilibrium for all $i \in \{2, \ldots, n\}$.  
We thus conclude that $H_{k+1}$ holds, which completes the proof.
\end{proof}

\begin{lemma}\label{lemma3}
Consider the system \eqref{eq:syst} under \Cref{assumption1}, \Cref{assumption2}, and \Cref{assumption3}. Then, any iterated Lie bracket involving $\boldsymbol{F}_2$ exactly once has a vanishing first component at the equilibrium point.
\end{lemma}
\begin{proof}
We proceed by induction. Let $H_k$ denote the property:  
\emph{“For any iterated Lie bracket of length $k$ that contains exactly one occurrence of $\boldsymbol{F}_2$, its first component vanishes at equilibrium.”}

\medskip
\noindent Initialization: 
For $k = 1$, by \Cref{assumption3}, we have $
F_2^{(1)}(\boldsymbol{p}^{\text{eq}}) = 0_{\mathbb{R}^n}$

\medskip
\noindent Induction step: 
Let $k \in \mathbb{N}$ and assume that $H_k$ holds.  
Consider an iterated Lie bracket $\boldsymbol{F}$ of length $k+1$ that contains exactly one occurrence of $\boldsymbol{F}_2$.  
We can write it as
\begin{equation*}
\boldsymbol{F} = [\boldsymbol{g}, \boldsymbol{h}],
\end{equation*}
where $\boldsymbol{g}$ and $\boldsymbol{h}$ are Lie brackets of length less than or equal to $k$.  
Moreover, exactly one of them contains $\boldsymbol{F}_2$; without loss of generality, assume it is $\boldsymbol{g}$.  
For simplicity, we omit the evaluation at $\boldsymbol{p}^{\text{eq}}$ in what follows.  
By definition of the Lie bracket, we have
\begin{equation*}
F^{(1)} = \sum_{j=1}^n \left( g^{(j)} \, \partial_j h^{(1)} - h^{(j)} \, \partial_j g^{(1)} \right).
\end{equation*}
By \Cref{lemma1}, $\boldsymbol{h}$ vanishes at equilibrium, and by the induction hypothesis, $g^{(1)}$ also vanishes at equilibrium.  
Hence,
\begin{equation*}
F^{(1)} = \sum_{j=2}^n g^{(j)} \, \partial_j h^{(1)}.
\end{equation*}
Applying \Cref{lemma2} to $\boldsymbol{h}$, we deduce that $\partial_j h^{(1)} = 0$ for all $j \ge 2$ at equilibrium.  
Therefore, $F^{(1)}$ vanishes at equilibrium, and thus $H_{k+1}$ holds.  
This completes the proof.
\end{proof}

\begin{proposition}\label{prop1}
    Consider the system \eqref{eq:syst}
 under \Cref{assumption1}, \Cref{assumption2} and \Cref{assumption3}. Then $R_1\subset\mathrm{Span}(\boldsymbol{e}_2, \ldots, \boldsymbol{e}_n)$.
 \end{proposition}
\begin{proof}
The result follows directly from the combination of \Cref{lemma1} and \Cref{lemma3}.
\end{proof}

\subsection{\texorpdfstring{Proof of \Cref{theorem:nonbstlc}}{Proof of Theorem 3}}

The linear map $AQB$ \eqref{eq:noninvertsyst} belongs to the set of matrices whose entries are analytic functions on $\mathbb{R}^2\times[0,2\pi]^3$. The family of vector fields $(\boldsymbol{F}_i)_{i=0}^2$ is obtained by multiplying by $(AQB)^{-1}$.  
Using symbolic computation (see \eqref{eq:detaqb}), the determinant $\det(AQB)$ admits a non-vanishing zeroth-order expansion around $0_{\mathbb{R}^5}$, which is strictly negative for all physically admissible parameters.  
Since the coefficients of $(AQB)^{-1}$ are obtained by multiplication and division of analytic functions and $\det(AQB)$ never vanishes near $0_{\mathbb{R}^5}$, the entries of $(AQB)^{-1}$ remain analytic in this neighborhood. Consequently, the family of vector fields $(\boldsymbol{F}_i)_{i=0}^2$ consists of analytic functions around $0_{\mathbb{R}^5}$.\\

\noindent Symbolic verification confirms that \Cref{assumption1} (as $\boldsymbol{F}_2(0)\neq 0$ for positive physical parameters), as well as \Cref{assumption2} and \Cref{assumption3}, are satisfied. Hence, by \Cref{prop1}, we have $R_1 \subset \mathrm{Span}(\boldsymbol{e}_2, \ldots, \boldsymbol{e}_5).$\\

\noindent \textbf{Case 1:} $\alpha \neq 0$.  
In this case, $\boldsymbol{F}_{212}(0) = 0$ and, according to \eqref{eq:F202}, we have $\boldsymbol{F}_{202}(0) \notin \mathrm{Span}(\boldsymbol{e}_2, \ldots, \boldsymbol{e}_5).$ Hence, $\boldsymbol{F}_{202}(0) \notin R_1 + \mathrm{Span}(\boldsymbol{F}_{212}(0)),$ and by applying \Cref{theorem32}, we deduce the desired result.\\

\noindent \textbf{Case 2:} $\alpha = 0$.  
Here, we have  $\boldsymbol{F}_{202}(0) = \boldsymbol{F}_{212}(0) = \boldsymbol{F}_{12,02}(0) = 0 \in R_1$  and $R' = 0.$  All Lie brackets appearing in \eqref{eq:Deqlamb} vanish except 
\begin{equation*}
\boldsymbol{F}_{02,002}(0) = \gamma \boldsymbol{e}_1.
\end{equation*}
Let $u_1^{\text{eq}} \in \mathbb{R}$. Then
\begin{equation*}
D_{u_1^{\text{eq}}}(\lambda_1, \lambda_2) = \lambda_1^2 \gamma \boldsymbol{e}_1.
\end{equation*}
By taking $\varphi = \gamma \boldsymbol{e}_1^{*}$ in the definition of \eqref{eq:CQ} and applying \Cref{theorem38}, we conclude that the system is not ($W^{1,\infty}$, B)-STLC.  \hfill $\blacksquare$

\section{Numerical Results}\label{sec:5}

\subsection{Local behavior near the equilibrium}
To numerically investigate the local behavior of the system around the equilibrium, we simulate trajectories corresponding to controls near the equilibrium control (which corresponds to the null control), as done in \cite{moreau2019}. Let $\varepsilon > 0$ be a small parameter. We define
\begin{equation}\label{eq:random_control}
\begin{array}{cc}
    u_1(t) = \varepsilon \left( \eta_1 + \eta_2 \cos(10t) + \eta_3 \cos(100t) \right), \\
    u_2(t) = \varepsilon \left( \eta_4 + \eta_5 \cos(10t) + \eta_6 \cos(100t) \right),
\end{array}
\end{equation}
where the coefficients $(\eta_i)_{i=1}^6$ are randomly sampled from the interval $[-1,1]$. By performing $N_{\text{MC}}$ realizations of these random controls, starting from the equilibrium state at the origin, we aim to identify, through such oscillatory control inputs, the regions that remain unreachable within a small time horizon $T$. The results are shown in \Cref{fig:traj_random_2_10_links} for the case of a $2$-link swimmer and for a $10$-link swimmer. When zooming in near the origin, a very thin area appears in both cases (when $x \to 0^{+}$) that remains unreachable.

\begin{figure}[tb]
    \centering
    \includegraphics[width=\linewidth]{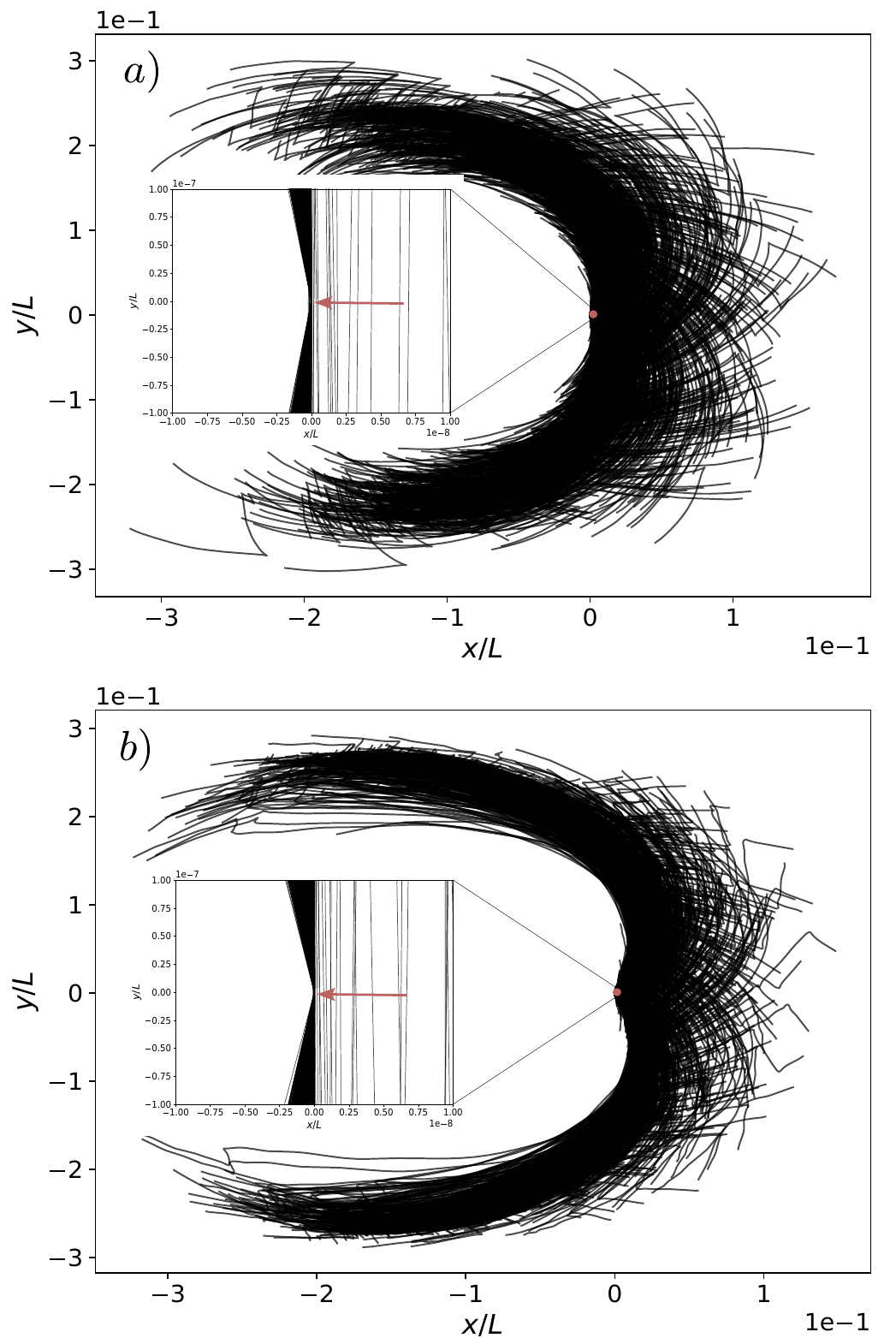}
    \caption{Trajectories around origin equilibrium for $N_{\text{MC}} = 2000$ realizations under the random oscillating control \eqref{eq:random_control}. $a)$ Trajectories of the $2$-link swimmer. $b)$ Trajectories of the $10$-link swimmer. A zoom on the neighborhood of the origin is shown for both plots. The parameters are taken from \Cref{table:swimmerparameters}, with $\varepsilon = 1$ and $T = 1$.
}
    \label{fig:traj_random_2_10_links}
\end{figure}

\subsection{Trajectory tracking}

Despite the lack of local controllability near equilibrium, the $2$-link swimmer can still be guided along desired trajectories. Let the reference trajectory be denoted by $\boldsymbol{p}_{\text{ref}}: [0,T] \to \mathbb{R}^2$,
with the final target state $\boldsymbol{p}_{\text{final}} = \boldsymbol{p}_{\text{ref}}(T)$. For a given control $\boldsymbol{u}$, we denote by $\boldsymbol{p}_{\boldsymbol{u}}$ the corresponding trajectory of the system. To achieve trajectory tracking, we formulate an optimal control problem whose cost functional is defined as
\begin{equation}\label{eq:cost_func}
J(\boldsymbol{p}_{\boldsymbol{u}},\boldsymbol{u})
    := \int_{0}^T \left\|\boldsymbol{p}_{\boldsymbol{u}}(t) - \boldsymbol{p}_{\text{ref}}(t)\right\|_{Q}^2 \, dt
    + \left\|\boldsymbol{p}_{\boldsymbol{u}}(T) - \boldsymbol{p}_{\text{final}}\right\|_{S}^2,
\end{equation}  
where $Q$ and $S$ are symmetric positive definite weighting matrices. The weighted norm is given by $\|\boldsymbol{p}\|_{Q}^2 = \boldsymbol{p}^\top Q \boldsymbol{p}$. The first term in \eqref{eq:cost_func} penalizes deviations from the reference trajectory along the time horizon, while the second term enforces precision at the terminal state. To solve this problem, we employ Bayesian optimization techniques, specifically the SCBO algorithm \cite{eriksson_scalable_2021}, which is particularly well-suited for large-scale optimization as used in \cite{palazzolo2025}. The control functions are parameterized using B-splines \cite{piegl1997}. A \emph{B-spline curve of degree $d$} is defined by
\begin{equation*}
    S(t):= \sum_{i=0}^N S_{i,d}(t) P_i,
\end{equation*}
where $S_{i,d}$ denotes the $i$-th B-spline basis function of degree $d$ defined over a \emph{knot vector} $\mathscr{T}$, and $P_i \in \mathbb{R}$ is the $i$-th control point. In what follows, we consider uniform knot vectors of the form
\begin{equation}\label{eq:knot_vector}
    \mathscr{T} = \{\underbrace{t_0,\ldots,t_0}_{d+1},\, t_1,\ldots,t_{n-1},\, \underbrace{t_n,\ldots,t_n}_{d+1}\},
\end{equation}
with equally spaced interior knots. The multiplicity of the first and last knots being $d+1$ ensures that the B-spline interpolates the first and last control points at $t=t_0$ and $t=t_n$. Since B-splines are entirely bounded by their control points, we define the following admissible control space:
{\small
\begin{equation}\label{eq:adm_control_nlinks_splines}
    \tilde{\mathscr{U}}_{\text{ad}}^{x,y} 
    := 
    \Biggl\{ \boldsymbol{u} = \left(\sum_{i=0}^{N_{u_1}} S_{i,d_1} P^{u_1}_i, \sum_{i=0}^{N_{u_2}} S_{i,d_2} P^{u_2}_i\right) \;\Bigg|\;
       |P^{u_1}_i| \leq M,\; |P^{u_2}_i| \leq M\Biggr\},
\end{equation}}
with $M=0.01$ (as in \cite{faris2020, oulmas2017}). Finally, the optimal control problem can be written as 
\begin{equation}\label{eq:pb_opti}
    \inf_{\boldsymbol{u}\in {\tilde{\mathscr{U}}_{\text{ad}}^{x,y}}} \; J(\boldsymbol{p}_{\boldsymbol{u}},\boldsymbol{u})
    \quad \text{subject to} (\boldsymbol{p}_{\boldsymbol{u},\boldsymbol{u})} \text{ solution  of } \eqref{eq:systplan}.
\end{equation} 

\noindent We look at a series of trajectories that are elliptical path segments of length $L$. Thus, the reference trajectory is defined by 
\begin{equation}\label{eq:ellipse_traj}
\boldsymbol{p}_{\text{ref}}(t) = \begin{bmatrix}a + a \cos(-\frac{t}{T} s_{\text{end}}+\pi)\\ b \sin(-\frac{t}{T}s_{\text{end}}+\pi)\\
0_{\mathbb{R}^3}\end{bmatrix},
\end{equation}
where the parameters $a$ and $b$ denote the semi-axes of the ellipse, and $s_{\text{end}}$ is chosen such that the total arc length satisfies
\begin{equation*}
L = \int_{0}^{s_{\text{end}}}\sqrt{a^2\sin^2(s) + b^2\cos^2(s)}\, ds.
\end{equation*}
The value of $s_{\text{end}}$ is determined numerically using a bisection method. We set $a = L$ and vary $b$ in $\{\tfrac{L}{2},\, L,\, \tfrac{3L}{2}\}$ in order to generate ellipses with different eccentricities. The final time is fixed at $T = 3$, and the controls are parameterized with $40$ control points per spline. The cost function \eqref{eq:cost_func} is evaluated using weighting matrices $Q = 10^9\diag(1,1,0,0,0)$ and $S = 10^4\diag(1,1,0,0,0)$. The resulting optimal trajectories are compared to those obtained using classical sinusoidal control, as proposed in \cite{Dreyfus2005, oulmas2017}. As illustrated in \Cref{fig:N_link_2D_ellipses}, the optimized trajectories exhibit significantly improved tracking performance, closely following the reference paths. Moreover, the optimal controls naturally adapt to the curvature of each reference trajectory, demonstrating the benefits of the B-spline–based flexible control design. Interestingly, the resulting optimal trajectories appear to display a quasi-periodic behavior.

\begin{figure}[tb]
    \centering
    \includegraphics[width=1\linewidth]{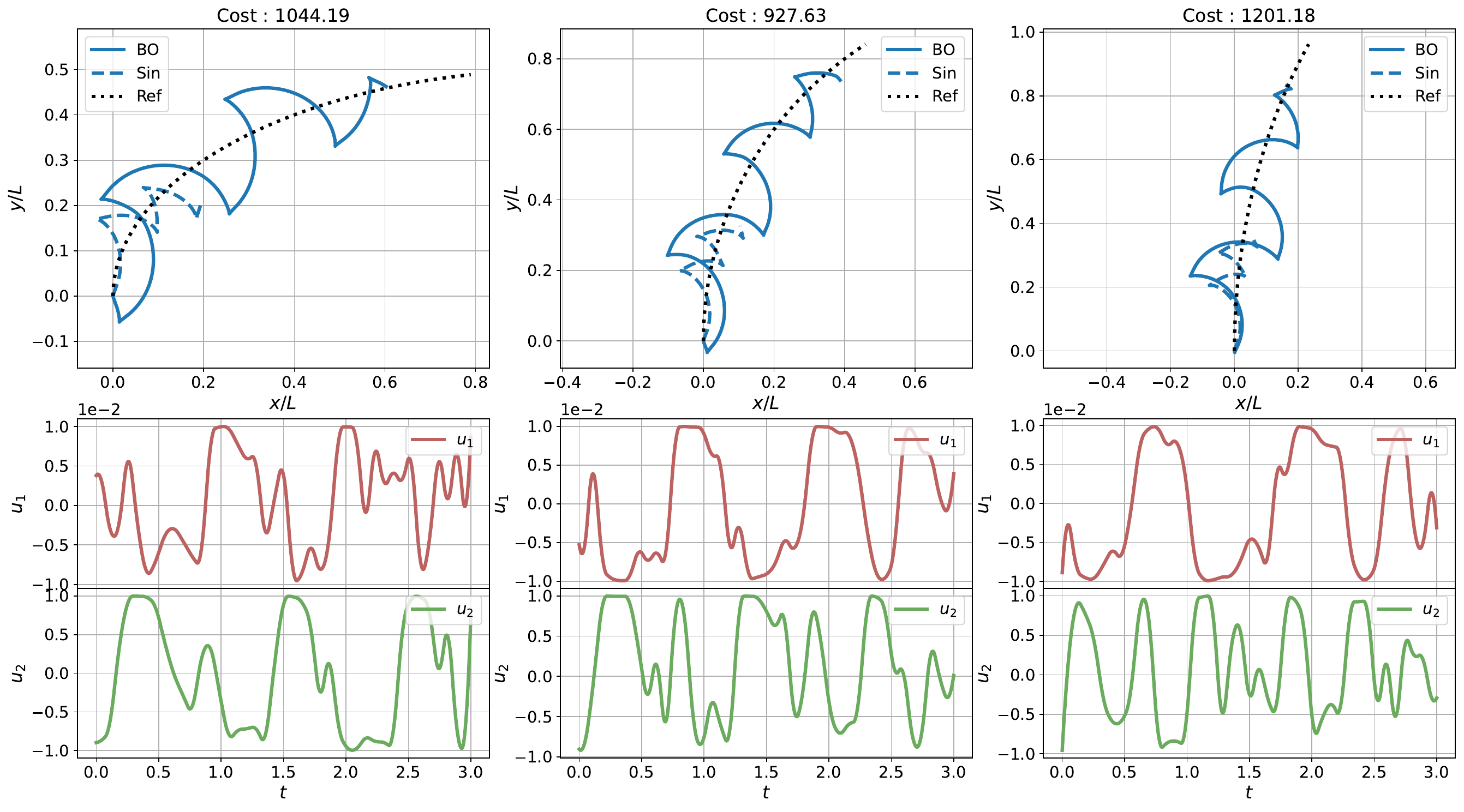}
    \caption{Trajectories of $2$-link swimmer for portion of elliptical references. From left to right: results for $b = \frac{L}{2}$, $b = L$, and $b = \frac{3L}{2}$. Top row: optimal trajectory (blue), trajectory under sinusoidal magnetic field of frequency $0.7Hz$ with tangent alignment (dashed line), and reference trajectory \eqref{eq:ellipse_traj} (black). Bottom row: optimized controls $u_1$ (red) and $u_2$ (green). The swimmer parameters are taken from \Cref{table:swimmerparameters}.}
    \label{fig:N_link_2D_ellipses}
\end{figure}

\section{Conclusions}\label{sec:6}

This study examined the controllability of a magnetic microswimmer composed of a magnetic head and an elastic flagellum. The analysis revealed that the swimmer is not small-time locally controllable in planar motion. Numerical simulations illustrated the existence of inaccessible regions near equilibrium and showed that, despite this limitation, trajectory tracking remains possible via Bayesian optimization. 

Future work will focus on two main directions. First, we aim to extend the theoretical analysis to three-dimensional trajectories, which would require adding an additional vector field to the dynamics, thus moving beyond the framework considered in \cite{moreau2024}. Second, we plan to investigate trajectory tracking in more complex environments, including the presence of boundaries or background flows, to better capture realistic microfluidic conditions.\\

\noindent \textbf{Supplemental Material:} The \texttt{Wolfram} code for computing symbolic expressions of the Lie brackets is available at the GitHub repository: \url{https://github.com/Luplz/MagneticSwimmer-LieBrackets}.



\section*{Appendices}

\subsection{\texorpdfstring{Coefficient of \eqref{eq:F202}}{Coefficient of (6)}}
The detailed expression of $\alpha$ introduced in \eqref{eq:F202} is given below:
\begin{equation}\label{eq:alpha}
    \alpha = \frac{\alpha_1}{\alpha_2}
\end{equation}
{\footnotesize
\begin{multline*}
    \alpha_1 = 324 k_{\text{el}}m^2\Big(132k_{\perp}^3l^4-4k_{\parallel}(k^h_{\perp})^2 r^2(5l^2-38lr+30r^2)\\+k_{\perp}^2(-132 k_{\parallel}l^4+l^2r(105 k^h_{\perp}l+47 k^h_{\parallel}l-252 k^h_{\perp}r))\\
    +2k_{\perp}k^h_{\perp}r(k^h_{\parallel}l(l-25r)r+2k_{\parallel}l^2(-38l+63r)\\+k^h_{\perp}r(9l^2-51lr+60r^2))\Big)
\end{multline*}
\vspace{-2em}
\begin{multline*}
    \alpha_2 = k_{\perp}l^3 (2k_{\parallel}l+k^h_{\parallel}r)\\\times\Big(14k_{\perp}^2l^4+42k^h_{\perp}k_r r^4+k_{\perp}lr(12k_r r^2+k^h_{\perp}(49 l^2+39lr+12r^2))\Big)^2
\end{multline*}
}

\subsection{Determinant of $AQB$} 

Using symbolic computation (see Supplemental Material), the zeroth-order series expansion of $\det(AQB)$ around $0_{\mathbb{R}^5}$ is
{\footnotesize
\begin{multline}\label{eq:detaqb}
\det(AQB) = -\frac{1}{216}\Big[k_{\perp}^2 l^6 (2k_{\parallel}l + k_{\perp}^h r)\\
\times\Big(14k_{\perp}^2 l^4 + 42k_{\perp}^h k_r r^4 + k_{\perp} l r \\
\times(12k_r r^2 + k_{\perp}^h(49l^2 + 39lr + 12r^2))\Big)\Big]\\ + O(\|\boldsymbol{p
}\|).
\end{multline}}

\subsection{Simulation parameters}
\begin{table}[h]
\caption{Parameters of the model used for the simulations, taken from \cite{faris2020} and fitted with an experimental swimmer.}
\label{table:swimmerparameters}
\begin{center}
\begin{tabular}{l l c}
\hline
\hline
Parameter & Symbol & Value\\
\hline
Length of the tail & $L$ & $7~\mathrm{mm}$\\
Radius of the head & $r$ & $0.3~\mathrm{mm}$\\
Parallel resistive coefficient (head) & $k^h_{\parallel}$ & $1.15~\mathrm{N\,s\,m^{-1}}$\\
Orthogonal resistive coefficient (head) & $k^h_{\perp}$ & $4.37~\mathrm{N\,s\,m^{-1}}$\\
Rotational resistive coefficient (head) & $k_r$ & $0.6~\mathrm{N\,s\,m}$\\
Parallel resistive coefficient (tail) & $k_{\parallel}$ & $0.35~\mathrm{N\,s\,m^{-1}}$\\
Orthogonal resistive coefficient (tail) & $k_{\perp}$ & $0.81~\mathrm{N\,s\,m^{-1}}$\\
Elastic coefficient & $k_{\text{el}}$ & $8.68\times10^{-7}~\mathrm{N\,m^{-1}}$\\
Magnetization & $m$ & $1.68\times10^{-4}~\mathrm{A\,m^{-1}}$\\
\hline
\hline
\end{tabular}
\end{center}
\end{table}

\section*{Acknowledgment}

The authors acknowledge the support of the French Agence Nationale de la Recherche
(ANR), under Grant No.
ANR-21-CE45-0013, Project NEMO. 

\addtolength{\textheight}{-12cm}   

\bibliographystyle{IEEEtran}
\bibliography{biblio}

\end{document}